\documentclass[12pt]{article} \textwidth=16.2cm \textheight=21cm
\usepackage{amsthm}
\usepackage{color}
\usepackage{cite}
\usepackage{graphicx}
\usepackage{psfrag}
\usepackage{url}
\usepackage{stfloats}
\usepackage{amsmath}
\usepackage{amssymb}
\usepackage{graphicx}
\usepackage[titletoc]{appendix}

\def\dref#1{(\ref{#1})}

 \def\dfrac{\displaystyle\frac}

\def\be{\begin{equation}}
\def\bel{\begin{equation}\label}
\def\ee{\end{equation}}
\def\ba{\begin{array}}
\def\ea{\end{array}}
\def\banl{\begin{eqnarray}\label}
\def\ean{\end{eqnarray}}

 \def\bna{\begin{eqnarray}}
\def\ena{\end{eqnarray}} \def\dref#1{(\ref{#1})}

\newtheorem{thm}{Theorem}
\newtheorem{lem}{Lemma}

\newtheorem{ass}{Assumption}

\def\dref#1{(\ref{#1})}

\begin{document}

\title{  On Instability of LS-based Self-tuning Systems with Bounded Disturbances  \footnotemark[1]
\author{Shuai Xu\footnotemark[2]
, \,Chanying Li \footnotemark[2]
          }
}

\footnotetext[1]{The work was supported by  the National Natural Science Foundation of China under Grants 61422308 and 11688101.  }

\footnotetext[2]{%
S.~Xu and C.~Li are with the Key Laboratory of Systems and Control, Academy of Mathematics and Systems Science, Chinese Academy of Sciences, Beijing 100190, P.~R.~China.   They are also with the School of Mathematical Sciences, University of Chinese Academy of Sciences, Beijing 100049, P. R. China.  Corresponding author: Chanying~Li (Email: \texttt{cyli@amss.ac.cn}).}

 \maketitle

\begin{abstract}
It is well known that discrete-time linear systems can be stabilized by a least-squares (LS) based self-tuning regulator (STR), as long as noises are absent. However, this note shows that  once the discrete-time linear systems are disturbed, the LS-based STR  is always running the risk of unstabilizing systems, no matter how small the noises are.

\end{abstract}


\section{Introduction}
It was proved early in \cite{goodwin1980discrete} that the following noise-free system
$$
A(q^{-1}) y_{t+1}=B(q^{-1})u_t
$$
can be stabilized by a least-squares based self-tuning regulator. In fact, as long as the noise is absent,  the LS-based STR even is capable of stabilizing the nonlinear  discrete-time system
\begin{equation}\label{systhe}
y_{t+1}=\theta^T f_t(y_t,\ldots, y_{t-p+1},u_{t-1},\dots,u_{t-q} ),
\end{equation}
whenever $\|f_t(x)\|=O(1)+O(\|x\|^b)$ with $b<8$ (see \cite{guo1996global}). But  what will happen if systems are disturbed by bounded noises? This may be more practical.
Relevant works in the stochastic framework  shed some light (e.g., \cite{aastrom1973self},
\cite{aastrom1977theory}, \cite{kumar1990convergence},  \cite{guo1991aastrom}, \cite{guo1995convergence}). \cite{guo1995convergence} studied the ARMA model, which is corrupted by a sequence of martingale difference noises, and derived the stability and optimality of the LS-based STR. Meanwhile,  the strong consistency of the LS estimator is guaranteed in the closed loop. Later on,
number $b=4$ and a polynomial criterion  have been put forward as the critical nonlinear characterizations  of the stabilizability   for systems in type \dref{systhe} but with random noises involved (see \cite{guo1997critical} and \cite{li2013stabilization}). Such systems can be stabilized by the LS-based STR,  when their nonlinearities are within the critical nonlinear conditions. Otherwise, no feedback control law is possible to stabilize them. This suggests that  noises play an  role here. The critical nonlinear growth rates are apparently  reduced by the involvement of noises.

 Now, a direct consequence of \cite{guo1991aastrom} indicates that if the noises are assumed to be bounded and i.i.d distributed, then with probability $1$, the LS-based STR can stabilize system 
$$
A(q^{-1}) y_{t+1}=B(q^{-1})u_t+C(q^{-1}) w_{t+1}.
$$
The trouble is, there still exist some sequences $\{w_t\}$ with probability $0$ such that the stochastic tools  could do nothing to them. 
 \cite{gusev1993algorithm}    observed the divergence of the LS estimator in a self-tuning system for some special bounded noises. Nevertheless, whether the LS-based self-tuning system is stable or not for bounded  noises  was still unknown to people yet.
We prove in this note that
there indeed exist some bounded noises that   will result in the instability of a LS-based self-tuning system, even for the simplest discrete-time linear model with a scalar unknown parameter.  Perhaps more surprisingly, as our result indicates,   once a discrete-time system is  disturbed, the LS-based STR  is always running the risk of unstabilizing it, no matter how small the noises are. In the meantime,  for the bounded  noises causing  the system unstable,  the LS estimator is proved to be  divergent during the closed-loop identification, as observed in \cite{gusev1993algorithm}.

Notably, though,
 discrete-time uncertain systems with bounded noises are stabilizable as well, provided their nonlinearities meet the polynomial criterion (see
 \cite{li2006robust} and \cite{li2006polynomial}). This means, different from the stochastic framework where the LS-STR converges to the minimal variance controller,  the LS-STR   in the deterministic framework  performs no longer ``optimal''.

\section{   Main Results  }
Consider  the  discrete-time single-input/single-output linear model:
\begin{align}\label{model}
 y_{t+1}=\theta y_t+u_t+w_{t+1},\quad t\geq 0,
\end{align}
where $y_t,u_t,w_t\in \mathbb{R}$ are the system output, input, and noise sequences, respectively.   Parameter $\theta\in \mathbb{R}$ is unknown.  Further, we assume

\begin{ass}
There is a number $w>0$ such that $|w_t| \leq w$ for all $t\geq 1$.
\end{ass}
The standard LS estimate $\theta_t$ of $\theta$  for model \dref{model} reduces to
 \begin{align}
 & \theta_{t+1}=~\theta_t+ \dfrac{1}{r_{t}} y_t (y_{t+1}- u_t-\theta_ty_t ),\label{a1}\\
 &\dfrac{1}{r_{t}} =\dfrac{1}{r_{t-1}} -\frac{y_t^2}{r_{t-1}^2+r_{t-1}y_t^2}, ~~r_0>0, \label{a2}
 \end{align}
where {$\theta_0,r_0$} are the deterministic initial values of the algorithm. The feedback law is designed according to the well-known certainty  equivalence principle:
\begin{align}\label{ut}
u_t=-\theta_t y_t.
\end{align}
Denote $\tilde{\theta}_t\triangleq\theta-\theta_t$, then the closed-loop system \dref{model} and \dref{ut} equals to
\begin{align}\label{1}
y_{t+1}=\tilde{\theta_t}y_t+w_{t+1}.
\end{align}

\begin{thm}\label{thm}
For any initial values $y_0$ and $\theta_0$, there exists a sequence $\{w_t\}$ satisfying Assumption 1  such that\\
(i) the LS estimate error $\tilde{\theta}_t$ diverges that
$\varlimsup_{t \to +\infty} |\tilde{\theta_t}| = +\infty $;\\
(ii) the outputs of the closed-loop system \dref{1}  satisfies
\begin{align}\label{4}
 \varlimsup_{t \to +\infty}\frac{ \sum\limits_{i=1}^t y_i^2}{\sum\limits_{i=1}^t w_i^2}= +\infty.
\end{align}
\end{thm}

\section{Proof of Theorem \ref{thm} }
The proof is based on several lemmas.
\begin{lem}\label{wtw}
Let  $k \geq 2$ be an integer that
\begin{align}\label{k}
r_{k-1} \geq w^2,~~|\tilde{\theta}_{k-1}| \leq 1,~~|y_{k-1}| \leq \frac{w}{2\sqrt{k-1}}.
\end{align}
Then,  $|w_t|\leq w$ for all $t\geq k$, if
 \begin{align}\label{wt}
w_t= -\tilde{\theta}_{t-1} y_{t-1} + \mathcal{S}(y_{t-1})\frac{w}{2\sqrt{t}},\quad t \geq k,
\end{align}
where
\begin{equation*}
\mathcal{S}(x) \triangleq
\left\{
\begin{array}{ll}
-1, &\mbox{if}~x < 0\\
1,  &\mbox{if}~x \geq 0
\end{array}.
\right.
\end{equation*}
\end{lem}

\begin{proof}
First of all, note that (\ref{a1}) and (\ref{a2}) yield
\begin{align}
\tilde{\theta}_{t+1}=\tilde{\theta_t}-\frac{y_t y_{y+1}}{r_t},  \label{2}
\end{align}
and  $r_t=\sum\limits_{i=1}^{t}{y_i^2}$ with $r_0=y_0^2$. According to \dref{1} and \dref{wt},
 we have
\begin{align}
y_t = \mathcal{S}(y_{t-1})\frac{w}{2\sqrt{t}},\quad t \geq k,\label{*}
\end{align}

We now verify that both $ w_{k}$ and $w_{k+1}$ are bounded by $w$. As a matter of fact,
by virtue of \dref{1}, \dref{k}  and \dref{*},
\begin{align*}
|w_{k}|
&\leq |\tilde{\theta}_{k-1} y_{k-1}| + |y_{k}| \\
&\leq \frac{w}{2} + \frac{w}{2} = w.
\end{align*}
Moreover,  \dref{2} yields
\begin{align*}
|\tilde{\theta}_{k}|
&\leq |\tilde{\theta}_{k-1}| + |\frac{y_{k-1}y_{k}}{r_{k-1}}|  \\
&\leq |\tilde{\theta}_{k-1}| + |\frac{y_{k-1}y_{k}}{w^2}| \\
&\leq 1+ \frac{1}{4} =\frac{5}{4},
\end{align*}
which by \dref{1} again,
\begin{align}
|w_{k+1}|
\leq |\tilde{\theta}_{k} y_{k}| + |y_{k+1}| \leq \frac{5}{4} ~\frac{w}{2\sqrt{2}} + \frac{w}{2} \leq w. \label{**}
\end{align}

When $t \geq k+1$,  $|y_t| \leq |y_{t-1}|$ due to (\ref{*}). In addition,  \dref{k} means $$r_{t-1} \geq r_{k-1} \geq w^2,$$
 so,  by (\ref{2}) and \dref{*},
\begin{align}\label{they-they}
|\tilde{\theta_t} y_t|-|\tilde{\theta}_{t-1} y_{t-1}|
~&\leq~ (|\tilde{\theta}_{t-1}|+\frac{y_{t-1} y_t}{r_{t-1}})|y_t|-|\tilde{\theta}_{t-1} y_{t-1}| \nonumber\\
~&\leq~ |\frac{y_{t-1} y_t^2}{r_{t-1}}| \leq~ \frac{w^3}{8r_{t-1}\sqrt{t-1}\sqrt{t}\sqrt{t}}\nonumber \\
~&\leq~ \frac{w}{8\sqrt{t-1}\sqrt{t}\sqrt{t}}
\end{align}
and
\begin{align}\label{y-y}
|y_t|-|y_{t+1}|
~&=~\frac{w}{2}(\frac{1}{\sqrt{t}}-\frac{1}{\sqrt{t+1}})\nonumber \\
~&=~\frac{w}{2\sqrt{t}\sqrt{t+1}(\sqrt{t}+\sqrt{t+1})}.
\end{align}
If $t = 2$, we have
\begin{align*}
4\sqrt{t-1}\sqrt{t} = 4\sqrt {2} \geq \sqrt{3}(\sqrt{2}+\sqrt{3}) = \sqrt{t+1}(\sqrt{t}+\sqrt{t+1}).
\end{align*}
For $t \geq 3$, it is also easy to compute
\begin{align*}
4\sqrt{t-1}\sqrt{t}
&= 2\sqrt{\frac{4}{3}t}\sqrt{3(t-1)} \\
&\geq 2\sqrt{t+1}\sqrt{t+1} \\
&\geq \sqrt{t+1}(\sqrt{t}+\sqrt{t+1}).
\end{align*}
Since integer $t \geq 2$, \dref{they-they} and \dref{y-y} shows
$$|\tilde{\theta_t} y_t|-|\tilde{\theta}_{t-1} y_{t-1}|  \leq |y_t|-|y_{t+1}|, $$
which yields
$$|\tilde{\theta_t} y_t|+|y_{t+1}| \leq |\tilde{\theta}_{t-1} y_{t-1}|+|y_t|.$$
As a consequence, by  (\ref{**}), we have
\begin{align*}
|w_t|&=|y_t-\tilde{\theta}_{t-1} y_{t-1}|\leq |\tilde{\theta}_{t-1} y_{t-1}|+|y_t| \\
     &\leq |\tilde{\theta}_{k} y_{k}|+|y_{k+1}| =w.
\end{align*}
The proof is completed.
\end{proof}

\begin{lem}
Let  $k \geq 2$ be an integer fulfilling \dref{k}.
If  $\{w_t\}$ satisfies \dref{wt},
then there is   an integer  $ k' \geq k$  such that $\{|\tilde{\theta_t}|, t \geq k'\}$ is a strictly increasing sequence and $\lim_{t \to +\infty} |\tilde{\theta_t}| = +\infty $.
\end{lem}

\begin{proof}
When  $\{w_t\}$ satisfies \dref{wt},  \dref{*} holds and as $t \to +\infty$,
$$r_{t+1} = r_{k} + \sum_{i=k+1}^{+\infty} y_{i}^{2} = r_{k} + \frac{w^{2}}{4} \sum_{i=k+1}^{+\infty} \frac{1}{i}\rightarrow +\infty.$$
So, as $t \to +\infty$,
$$\prod_{i=k+1}^{t} (1+\frac{y_{i+1}^2}{r_i}) = \prod_{i=k+1}^{t} \frac{r_{i+1}}{r_i} = \frac{r_{t+1}}{r_{k+1}} \to +\infty,$$
which immediately shows
$$ \sum_{i=k+1}^{+\infty}\frac{y_{i+1}^2}{r_i} = +\infty.$$

Now, since $y_t y_{t+1} \geq 0$ for all $t \geq k+1$, by \dref{*},
$$\sum_{t=k+1}^{+\infty} \frac{y_t y_{t+1}}{r_t} \geq \sum_{t=k+1}^{+\infty} \frac{y_{t+1}^2}{r_t} = +\infty.$$
In view of  (\ref{2}),  $\tilde{\theta}_t$ decreases monotonically and
 $\lim_{t \to +\infty} \theta_t = -\infty.$
 Let $k'\triangleq\min\{t: \tilde{\theta}_t<0\},$ then $\{|\tilde{\theta}_t|, t \geq k'\}$ is a strictly increasing sequence.
\end{proof}

\begin{lem}\label{lrj}
Let $j\geq 0$ be an integer that $y_j \neq 0,~\tilde{\theta}_j \neq 0$ and $w_t=0$ for all $t \geq j+1$. The following two statements hold: \\
(i) there are some numbers $c_0>0$, $ \alpha \in (0, 1)$ and $t_0\geq j$ such that
$$ |y_t| \leq c_0\alpha^{t},\,\,t\geq t_0 \quad \mbox{and}\quad 0 < \lim\limits_{t \to +\infty} |\tilde{\theta}_t| < 1;$$
\\
(ii) if $|\tilde{\theta}_j| \geq 6$, then there is an integer $l \geq j$ such that
\begin{align}\label{r3r}
r_{l-1} \leq 3r_{j-1}\quad \mbox{and}\quad y_l^2 \geq r_{j-1}.
\end{align}
\end{lem}

\begin{proof}
We now prove statement (i). By (\ref{1}) and (\ref{2}),
\begin{align}
r_{t-1} \tilde{\theta_t} = r_0 \tilde{\theta_1} -  \sum\limits_{i=1}^{t}{y_{i-1}}w_i. \label{3}
\end{align}
From  (\ref{1}), (\ref{2}) and (\ref{3}), for all $t \geq j$,
\begin{align}
&y_{t+1}=\tilde{\theta}_t y_t, \label{5}\\
&\tilde{\theta}_{t+1}=\frac{r_{t-1}}{r_t} \tilde{\theta}_t, \label{6}\\
&r_{t-1} \tilde{\theta}_t = r_{j-1} \tilde{\theta}_j. \label{7}
\end{align}

Since $\{r_t\}$ is an increasing sequence,   (\ref{6}) shows that $\{|\tilde{\theta}_t|, t\geq j\}$ is a decreasing sequence. So, $|\tilde{\theta}_t|$ is bounded for all $t\geq j$ and  $\lim_{t\to +\infty} |\tilde{\theta}_t|$  exists. We assert $\lim_{t \to +\infty} |\tilde{\theta}_t| < 1$. Otherwise,
if $\lim_{t \to +\infty} |\tilde{\theta}_t| \geq 1$, then $|\tilde{\theta}_t| \geq 1$ for all $j\geq t$.
By (\ref{5}), we have
\begin{align}
r_t&=r_{j-1}+y_j^2+y_{j+1}^2+\dots+y_t^2 \nonumber\\
   &=r_{j-1}+y_j^2+\tilde{\theta}_j^2 y_j^2+\dots+ (\prod_{i=j}^{t-1} \tilde{\theta}_i^2)  y_j^2 \label{rty}\\
   &\geq r_{j-1}+(t-j+1)y_j^2.\nonumber
\end{align}
This immediately leads to  $\lim_{t \to +\infty} r_t=+\infty$. By  $\lim_{t \to +\infty} |\tilde{\theta}_t| \geq 1$ again, $$\lim_{t \to +\infty} r_{t-1} \tilde{\theta}_t=+\infty,$$
 which contradicts to  (\ref{7}). Therefore,  $\lim_{t \to +\infty} |\tilde{\theta}_t| < 1$ and hence, $|\tilde{\theta}_{t_0}| < 1$ for some integer $t_0\geq j$.

 When $t \geq t_0\geq j$, by the fact that $\{|\tilde{\theta}_t|, t\geq j\}$ is a decreasing sequence,  (\ref{5}) shows that $$|y_t| \leq |y_{t_0}||\tilde{\theta}_{t_0}|^{t-{t_0}}.$$
 The first formula of (i) is thus derived by letting
 $c_0=|y_{t_0}||\tilde{\theta}_{t_0}|^{-{t_0}}$ and $\alpha=|\tilde{\theta}_{t_0}|\in (0,1)$. So, as $t\to +\infty$,
 \begin{equation}\label{boundr}
 r_t=\sum\limits_{i=1}^{t}{y_i^2}=O(1)+\sum_{i=t_0}^t c\alpha^{i}=O(1).
 \end{equation}
 This together with \dref{7} infers that $\lim_{t \to +\infty} |\tilde{\theta}_t| > 0$.

To prove statement (ii), we first assert that there is an integer  $l\geq j$ that $r_l \geq 3r_{j-1}$. Otherwise,
 $r_{t-1} < 3r_{j-1}$ for all $ t \geq j$.  Therefore, by  $|\tilde{\theta}_j| \geq 6$ and  (\ref{7}), we have for any $t \geq j$,
 \begin{equation}\label{thet1/3}
 |\tilde{\theta}_t| > |\frac{1}{3}\tilde{\theta}_j| \geq 2.
 \end{equation}
Consequently, by \dref{rty}, as $t\to +\infty$,
\begin{align*}
r_t&=r_{j-1}+y_j^2+\tilde{\theta}_j^2 y_j^2+\dots+ (\prod_{i=j}^{t-1} \tilde{\theta}_i^2)  y_j^2 \\
   &>r_{j-1}+ y_j^2+2y_j^2+\dots+2^{t-j}y_j^2\rightarrow+\infty,
\end{align*}
which contradicts to \dref{boundr}. Let $$l\triangleq \min\{t\geq j: r_t \geq 3r_{j-1} \},$$ then  $r_{l-1} \leq 3r_{j-1}$ and $r_l \geq 3r_{j-1}$.

The remainder is devoted to showing  $y_{l}^2 \geq r_{j-1}$. Otherwise, if $y_{l}^2 < r_{j-1}$,
by noting that $r_{t-1} \leq 3r_{j-1}$ for any $j+1 \leq t \leq l$,  \dref{thet1/3} holds as well  for all $ t\in [j+1,l]$ and hence
 $ \frac{1}{|\tilde{\theta}_t|} <\frac{1}{2}$.
Rewrite $r_t$ by
\begin{align*}
r_{l}&=r_{j-1}+y_j^2+\dots+y_{l-1}^2+y_{l}^2 \\
         &=r_{j-1}+\frac{1}{\prod\limits_{i=j}^{l-1} \tilde{\theta}_i^2} y_{l}^2+\dots+\frac{1}{\tilde{\theta}_{l-1}^2} y_{l}^2+y_{l}^2 \\
         &< r_{j-1}+\frac{1}{2^{l-j}} y_{l}^2+\dots+\frac{1}{2} y_{l}^2+y_{l}^2 < 3r_{j-1},
\end{align*}
which derives a contradiction. (ii) is thus proved.
\end{proof}

\begin{lem}\label{j>c}
Let  $k \geq 2$  satisfy \dref{k}. Given a constant $c>0$, set
\begin{equation}\label{wk'j}
w_t=
\left\{
\begin{aligned}
& -\tilde{\theta}_{t-1} y_{t-1} + \mathcal{S}(y_{t-1})\frac{w}{2\sqrt{t}},~~~~k \leq t \leq j\\
&0 ,~~~~t \geq j+1\\
\end{aligned}
\right.
\end{equation}
where $j\geq k$ is an integer such that
\begin{align}
|\tilde{\theta}_j| \geq \max\{|\tilde{\theta}_0|,|\tilde{\theta}_1|,\dots,|\tilde{\theta}_{j-1}|,6, 126c\}. \label{8}
\end{align}
Then,
$( \sum_{i=1}^{+\infty} y_i^2)/(\sum_{i=1}^{+\infty} w_i^2)\geq c. $

\end{lem}

\begin{proof}
By  (\ref{1}), \dref{k}, \dref{wk'j} and (\ref{8}), we have
\begin{align*}
\sum\limits_{i=1}^{+\infty} w_i^2 &= \sum\limits_{i=1}^{j} w_i^2
= \sum\limits_{i=1}^{j} (y_i-\tilde{\theta}_{i-1} y_{i-1})^2 \\
&\leq 2\sum\limits_{i=1}^{j}y_i^2 + 2\sum\limits_{i=1}^{j} \tilde{\theta}_{i-1}^2 y_{i-1}^2 \\
&\leq 2(r_{j-1}+y_j^2) + 2r_{j-1}\tilde{\theta}_j^2 \\
&\leq 2r_{j-1} + 2(2y_{j-1}^2 \tilde{\theta}_{j-1}^2 + 2w^2) + 2r_{j-1}\tilde{\theta}_j^2 \\
&\leq 2r_{j-1} + 4 w^2 + 4r_{j-1} \tilde{\theta}_{j}^2  + 2r_{j-1}\tilde{\theta}_j^2 \leq 7r_{j-1}\tilde{\theta}_j^2.
\end{align*}
Note that $w_t=0$ for all $t\geq j+1$,
in view of Lemma \ref{lrj}, there exists an integer $l \geq j$ fulfilling \dref{r3r}. By  (\ref{7}), $r_{l-1} \leq 3r_{j-1}$ yields $|\tilde{\theta}_l|^2 \geq \frac{1}{9} |\tilde{\theta}_j|^2$. Then,  \dref{r3r}, (\ref{5}) and (\ref{7}) imply
\begin{align*}
y_{l+1}^2+y_{l+2}^2
&= \tilde{\theta}_l^2 y_l^2 (1+\tilde{\theta}_{l+1}^2) \geq 2\tilde{\theta}_l^2 y_l^2 |\tilde{\theta}_{l+1}| \\
&\geq \frac{2}{9} \tilde{\theta}_j^2 y_l^2 \frac{r_{j-1}|\tilde{\theta}_j|}{r_{l-1}+y_l^2} \geq \frac{2}{9} r_{j-1} |\tilde{\theta}_j|^3 \frac{y_l^2}{3r_{j-1}+y_l^2} \\
&\geq \frac{1}{18} r_{j-1} |\tilde{\theta}_j|^3.
\end{align*}
Moreover, $|\tilde{\theta}_j| \geq 126c$, the above inequality shows
$$\frac{ \sum\limits_{i=1}^{+\infty} y_i^2}{\sum\limits_{i=1}^{+\infty} w_i^2} \geq \frac{y_{l+1}^2+y_{l+2}^2}{7r_{j-1}\tilde{\theta}_j^2} \geq c, $$
which completes the proof.
\end{proof}

\begin{proof}[\textbf{Proof of Theorem 1}]
Set the noise
\begin{equation*}
w_t=
\left\{
\begin{array}{ll}
\mathcal{S}(\tilde{\theta}_0 y_0)w, &t=1\\
0,&2 \leq t \leq k_1-1
\end{array},
\right.
\end{equation*}
where for $\mathcal{K}_1\triangleq\{k\geq 3:   \dref{k} \mbox{ holds} \}$,
\begin{equation}
k_1\triangleq
\left\{
\begin{array}{ll}
\min_{k\in \mathcal{K}_1} k,&  \mathcal{K}_1\neq \emptyset \\
+\infty, &  \mathcal{K}_1= \emptyset
\end{array}.
\right.
\end{equation}
Clearly, $|w_t|\leq w$ for $ t\in [1,k_1-1]$.
We next claim that  $k_1$ is finite. Otherwise, $\mathcal{K}_1=\emptyset$. Then,  \dref{k} fails for every $k \geq 3$ and $w_t=0$ whenever $t\geq 2$.
Now,  $w_1= \mathcal{S}(\tilde{\theta}_0 y_0)w$, which means $|w_1|=w$ and $y_1^2 \geq w^2>0$.  Therefore,
\begin{equation}\label{rk>w}
r_{k-1} \geq r_1\geq w^2,\quad \forall k\geq 3.
\end{equation}
If $\tilde{\theta}_1=0$,  it is easy to compute that $\tilde{\theta}_2=y_2=0$ due to $w_2=0$. So, \dref{k} holds for $k=3$. This asserts
 $\tilde{\theta}_1 \neq 0$.     Consequently, by   Lemma 3(i), there are some $c_0>0$ and $\alpha\in(0,1)$ such that for all
sufficiently large $k\geq 3$,
 \begin{equation} \label{k-1c0}
|y_{k-1}| \leq c_0\alpha^{k-1}\leq   \frac{w}{2\sqrt{k-1}}\quad \mbox{and}\quad |\tilde{\theta}_{k-1}| < 1,
\end{equation}
which  together with \dref{rk>w}  contradicts to $\mathcal{K}_1=\emptyset$. Therefore,   $k_1$ is finite.

Now, fix $k_1\in \mathcal{K}_1$. Define
\begin{equation} \label{j_1}
j_1\triangleq
\left\{
\begin{array}{ll}
\min_{j \in \mathcal{J}_1  } j,  & \mathcal{J}_1 \neq \varnothing\\
+\infty, & \mathcal{J}_1=\varnothing
\end{array},
\right.
\end{equation}
where
\begin{equation*}
\mathcal{J}_1\triangleq\{j\geq k_{1}:|\tilde{\theta}_j| \geq \max\{|\tilde{\theta}_0|,|\tilde{\theta}_1|,\dots,|\tilde{\theta}_{j-1}|, 126\} \}.
\end{equation*}
Moreover, let
\begin{equation*}
k_2\triangleq
\left\{
\begin{array}{ll}
\min_{k \in \mathcal{K}_2 }{k},  & \mathcal{K}_2 \neq \varnothing\\
+\infty, & \mathcal{K}_2=\varnothing
\end{array}
\right.
\end{equation*}
with
\begin{equation*}
\mathcal{K}_2\triangleq\left\{k\geq j_{1}+2: \dref{k} \mbox{ holds and } \sum\nolimits_{i=1}^{k-1} y_i^2 \geq \sum\nolimits_{i=1}^{k-1} w_i^2\right\}.
\end{equation*}
For $j_1$ and $k_2$ defined above, set
\begin{equation*}
w_t=
\left\{
\begin{array}{ll}
 -\tilde{\theta}_{t-1} y_{t-1} + \mathcal{S}(y_{t-1})\dfrac{w}{2\sqrt{t}},&k_1 \leq t \leq j_1\\
0,&j_1+1 \leq t \leq k_2-1
\end{array}.
\right.
\end{equation*}
Noting that \dref{k} is true for $k=k_1$,  by Lemma \ref{wtw},
$
|w_t|\leq w$ for all $t\in [k_1,k_2-1].
$

We proceed to prove that both $j_1$ and $k_2$ are finite. If $j_1=+\infty$, then $w_t$ satisfies \dref{wt} for all $t\geq k_1$. Further,
since  \dref{k} holds for $k=k_1$, by  Lemma 2, $\{|\tilde{\theta}_t|, t\geq k'_1\}$ is an increasing sequence for some $k'_1\geq k_1$ and $\lim_{t \to +\infty} |\tilde{\theta_t}| = +\infty $, which gives $\mathcal{J}_1 \neq \emptyset$. Hence $j_1$ is finite or a contradiction arises. So, by Lemma \ref{j>c}, we immediately deduce that for all sufficiently large $k$,
\begin{equation}\label{wk2}
\sum\limits_{i=1}^{k-1} y_i^2 \geq \sum\limits_{i=1}^{k-1} w_i^2.
\end{equation}
It is clear that $\tilde{\theta}_{j_1}\neq 0$ as $j_1\in \mathcal{J}_1$ and  $y_{j_1} \neq 0$ by \dref{*}. Similar to the proof of $k_1<+\infty$,
Lemma \ref{lrj}(i) shows that $k_2$ is finite.

Suppose two increasing sequences $\{k_i\in \mathbb{N}^+, 1 \leq i \leq s\}$, $\{j_{i}\in [k_i, k_{i+1}-2]\cap \mathbb{N}^+,  1 \leq i \leq s-1\}$  and a series   $\{w_t,1\leq t\leq k_s-1\}$  are  constructed for some $s\geq 2$  such that \dref{k} holds for $k=k_s$, $|\tilde{\theta}_{j_{s-1}}| \geq 126(s-1)$   and
 \begin{equation}\label{yws}
 \sum\limits_{i=1}^{k_s-1} y_i^2 \geq (s-1) \sum\limits_{i=1}^{k_s-1} w_i^2.
 \end{equation}
Analogous arguments of \dref{k-1c0} and \dref{j_1}--\dref{wk2}  yield that there are two finite integers  $k_{s+1}$ and $ j_s$, as well as  a sequence
$\{w_t, k_s\leq t\leq k_{s+1}-1\}$ such that  \dref{k} holds for $k=k_{s+1}$, $|\tilde{\theta}_{j_{s}}| \geq 126s$   and
 $$\sum\limits_{i=1}^{k_{s+1}-1} y_i^2 \geq s \sum\limits_{i=1}^{k_{s+1}-1} w_i^2.$$
 So, there  exists a $\{w_t, t\geq 1\}$ and a $\{(k_s, j_s), s\geq 1\}$  fulfilling  \dref{yws} and $|\tilde{\theta}_{j_s}| \geq 126s$   for each $s\geq 2$. Statements (i) and (ii) are thus derived as desired.
\end{proof}



\end{document}